\documentclass[a4paper, reqno, 12pt]{amsart}
\usepackage[english]{babel}
\usepackage[T1]{fontenc}
\usepackage{mathptmx,amsmath}
\usepackage{hyperref}
\usepackage{graphicx}
\usepackage{amssymb,verbatim,cases,mathtools}
\usepackage{fullpage}
\usepackage[svgnames,hyperref]{xcolor}

\newtheorem{theorem}{Theorem}[section]
\newtheorem{corollary}[theorem]{Corollary}
\newtheorem{lemma}[theorem]{Lemma}

\newtheorem{proposition}[theorem]{Proposition}
\theoremstyle{remark}
\newtheorem{question}[theorem]{Question}
\newtheorem{remark}[theorem]{Remark}

\numberwithin{equation}{section}
\newcommand{\B}{\mathbb{B}}

\newcommand{\E}{\mathcal{E}}
\newcommand{\Em}{\mathcal{E}_m}
\newcommand{\Eom}{\mathcal{E}^0_m}
\newcommand{\F}{\mathcal{F}}
\newcommand{\Fm}{\mathcal{F}_m}

\newcommand{\SHm}{SH_m}

\newcommand{\n}{\noindent}
\newcommand{\fr}{\frac}
\newcommand{\ve}{\varepsilon}
\newcommand{\la}{\lambda}
\newcommand{\psh}{\text{plurisubharmonic}}

\newcommand{\vol}{\text{vol}}

\newcommand{\de}{\delta}

\newcommand{\nhd}{neighbourhood}

\newcommand{\om}{\omega}
\newcommand{\Om}{\Omega}

\begin{document}
\author{Nguyen Quang Dieu \textit{$^{1,2}$} and Do Thai Duong \textit{$^{3}$} }
\address{\textit{$^{1}$}~Department of Mathematics, Hanoi National University of Education, 136 Xuan Thuy, Cau Giay, Hanoi, Vietnam}
\address{\textit{$^{2}$}~Thang Long Institute of Mathematics and Applied Sciences,
	Nghiem Xuan Yem, Hoang Mai, HaNoi, Vietnam}
\email{dieu$\_$vn$@$yahoo.com, ngquang.dieu@hnue.edu.vn}
\address{\textit{$^{3}$}~Institute of Mathematics, Vietnam Academy of Science and Technology, 18 Hoang Quoc Viet, Cau Giay, 100000 Hanoi, Vietnam} 
\email{dtduong@math.ac.vn}
\title{Decay near boundary of volume of sublevel sets of $m-$subharmonic functions}
%\thanks{The research of the author was supported VNU}
\date{\today}
\maketitle
\begin{abstract}
We investigate decay near  boundary of the volume of sublevel sets in Cegrell classes of $m-$ subharmonic function on bounded domains in $\mathbb C^n.$ On the reverse direction, some sufficient conditions for membership in certain Cegrell's classes, in terms of the decay of the sublevel sets, are also discussed.
\end{abstract} 
\tableofcontents
\section{Introduction}
Let $\Om$ be a domain in $\mathbb C^n$ and let $u$ be a subharmonic function defined on $\Om$. Then, for an integer $m, 1\leq m\leq n$, according to  Li in \cite{Li}, we say that $u$ is $m-$subharmonic function if for every
$\alpha_1,...,\alpha_{m-1}\in\Gamma_m$, the inequality
$$dd^cu\wedge\alpha_1\wedge...\wedge\alpha_{m-1}\wedge\omega^{n-m}\geq0,$$
holds in the sense of currents. Here we define
$$\Gamma_m:=\{\alpha\in C_{(1,1)}:\alpha\wedge\omega^{n-1}\geq0,...,\alpha^m\wedge\omega^{n-m}\geq0\},$$
where $\omega:=dd^c|z|^2$ is the canonical K\"ahler form in $\mathbb C^n$ and $C_{(1,1)}$ is the set of $(1,1)-$forms with constant coefficients. Denote by $\SHm(\Omega)$ the set of all $m-$subharmonic	functions in $\Omega$, and $\SHm^-(\Omega)$ for the set of all non-positive $m-$subharmonic functions in $\Omega$. The following chain of inclusions is then obvious
$$PSH=SH_n\subset...\subset SH_1=SH.$$
The border cases, $SH_1$ and $SH_n$, of course,  correspond to subharmonic function and plurisubharmonic functions which are of fundamental importance in potential theory and pluripotential theory respectively.
Later on, using Bedford-Taylor's induction method in \cite{BT},
Blocki extended the definition of the complex $m-$Hessian operator $(dd^cu)^m\wedge\omega^{n-m}$ to locally bounded $m$-subharmonic functions in \cite{Blo05}. 	In particular, if $u \in\SHm(\Omega)\cap L^\infty_{loc}(\Omega)$ then the Borel measure 
$(dd^cu)^m\wedge\omega^{n-m}$
is well-defined and is called the complex $m-$Hessian of $u$.

\n
More recently, in \cite{Lu15}, Lu following the framework of Cegrell (in \cite{Ceg98} and \cite{Ceg04}) studied the domain of existence for the complex $m-$Hessian operator. For this purpose, he introduced finite energy classes of m-subharmonic functions of Cegrell type on bounded $m-$ hyperconvex domains $\Om$,
i.e., domains that admit a negative $m-$subharmonic exhaustion function
\begin{align*}
\Eom(\Om)&=\{u\in\SHm^-(\Om)\cap L^\infty(\Om):\lim\limits_{z\rightarrow\partial\Om}u(z)=0,\ \int_{\Om}(dd^cu)^m\wedge\omega^{n-m}<\infty\},\\
\F_m(\Om)&=\{u\in\SHm^-(\Om):\exists \Eom(\Om)\ni u_j\downarrow u,\sup_j \int_{\Om}(dd^cu_j)^m\wedge\omega^{n-m}<\infty\},\\
\Em(\Om)&=\{u\in\SHm^-(\Om):\forall G\Subset\Om,\ \exists u_G\in\F_m(\Om)\text{ such that }u=u_G\text{ on }G\}.
\end{align*}
Then the complex $m-$Hessian operator can be defined on the class $\Em(\Om)$. Moreover, this is the largest subset of non-positive $m$-subharmonic functions defined on $\Om$ for which the complex $m-$Hessian operator can be continuously extended. The reader is also referred to \cite{DBH} for another solid development of $m-$Hessian operator.

Our work is inspired partly by some recent results in \cite{NVT18} where the author characterizes the classes $\Em$, $\F_m$ in terms of the $m-$capacity of sublevel sets. Notice that similar result for the case of $m=n$
was obtained much earlier in  Section 3 of \cite{CKZ}.

The aim of this paper is to study behavior near boundary of {\it volume} of sublevel sets of the class $\F_m$. Our first result gives some qualitative estimates on portion near the boundary of the sublevel sets of
$u \in \F_m$.

\medskip
\noindent
\textbf{Theorem A.}
{\it Let $\Omega$ be a bounded $m-$hyperconvex domain in $\mathbb C^n,$ $\rho \in \Em(\Om)$ and $u \in \Fm (\Om).$ For $\ve, \de>0$ we
set $$\Om_{u,\ve,\de}:=\{z \in \Om: u(z)<-\ve, \rho(z)>-\de\}.$$
Then we have the following estimates:

\n 	
(a)
$\int\limits_{\Om_{u,\ve,\de}} (dd^c \rho)^m\wedge\omega^{n-m} \le \big (\fr{\de}{\ve} \big)^m\int\limits_{\Om} (dd^c u)^m\wedge\omega^{n-m}.$
	
\n 
(b) ${\big (\fr{m}{m+1}}\big )^{m+1}\int\limits_{\Om_{u,\ve,\de}} d\rho \wedge d^c \rho \wedge (dd^c \rho)^{m-1}\wedge\omega^{n-m} \le \de \big (\fr{\de}{\ve} \big)^m \int\limits_{\Om} (dd^c u)^m\wedge\omega^{n-m}$, if $\rho$ is locally bounded.}

\n 
The proof of Theorem A uses a version of a classical comparison principle due to Bedford and Taylor in \cite{BT} but for $m-$subharmonic functions, and of course the structure of Cegrell's classes that involved.
Under stronger convexity assumptions on $\Om$ we are able to derive upper bounds for 
volume of $\Om_{u,\ve,\de}$ that depend on $\ve, \de$ and the total $m-$ Hessian measure of $u$ (cf. Corollary 3.2 and Corollary 3.3)

\n
Using the same technique and a subextension result for $m-$ subharmonic functions coupled with a symmetrization trick, we prove the second main result which estimates the volumes of the sublevel sets near certain boundary points of $\Om.$ 

\n
\noindent
\textbf{Theorem B.}
{\it Let $\Om$ and $u$ be as in Theorem A and $\xi \in \partial \Om.$ Let $\eta \in \mathbb C^n$ be a point such that
$$\vert \xi-\eta\vert=d(\eta):= \sup \{\vert z-\eta \vert: z \in \overline{\Om}\}.$$
Then for all $\de \in (0,d(\eta))$ and $t>0$ we have
$$\begin{aligned}
&\vol_{2n} \{z \in \Om: u(z)<-t, d(\eta)-\delta <\vert z-\eta \vert <d(\eta)\}\\
&\le a_n d^{2n-2}\big (\int\limits_{\Om} (dd^c u)^m\wedge\omega^{n-m}\big )^{1/m}
\fr{\de^2}{t},
\end{aligned}$$
where $d$ is the diameter of $\Om$ nd $a_n>0$ is a constant depending only on $n.$}
\begin{remark}
For a given $\xi$, there may exists no point $\eta \in \partial \Om$ such that $\vert \xi-\eta \vert=d(\eta).$ Indeed, any point $\xi$ in the inner sphere of the annulus
$\{r<\vert z\vert<1\}$ $(r \in (0,1))$ does not have this property.
\end{remark}

In case $\Omega$ is the unit ball $\B^{n}$ in $\mathbb C^n$, by taking $\xi$ to be an arbitrary point in 
$\partial \B^n$ and letting $\eta$ be the origin in Theorem B, we obtain  
the following result.

\medskip	
\noindent
\textbf{Corollary C.}
{\it Let $u\in\F_m(\B^{n})$. Then there exists $C>0$ such that for $A>0$ we have
$$\limsup\limits_{\de \rightarrow0^+}\frac{\vol_{2n} \{z \in \B^{n}: u(z)<-A\de,\ \|z\|>1-\de\}}{\de}<\frac{C}{A}.$$
}
\n 
Observe that the above result in the case $m=n$ was proved in Theorem 5 in \cite{DD19}.
Our next main result is a sufficient condition for membership of the class $\F_m$ in the case when $\Om$ admits
a nice defining $m-$subharmonic function.

\noindent
\textbf{Theorem D.}
{\it Let $\Omega$ be a bounded $m-$hyperconvex domain in $\mathbb C^n$ that admits a negative $m-$subharmonic exhaustion function $\rho$ which is $\mathcal C^1-$smooth on a \nhd\ of $\partial \Om$ and satisfies $d\rho \ne 0$ on $\partial \Om.$ Let $u\in SH_m^{-} (\Omega)$ be such that there exist $A,C>0$ and $\alpha>2n$ satisfying $$\vol_{2n}(\{z\in\Omega: \text{d}(z,\partial\Omega)<\de,\ u(z)<-A\de\})\leq C \de^\alpha,$$
for all $\ve>0$ small enough. Then $u\in\F_m(\Omega)$.}

\n 
The proof proceeds roughly as follows. First by averaging $u$ over small balls, we may approximate $u$ from above by a sequence $u_\ve$ of $m-$subharmonic functions defined on slightly smaller domains than $\Om.$ Then, by the assumptions of the theorem we can glue each $u_\ve$ with a suitable defining function for $\Om$ to obtain an element in $\E^-_m (\Om)$ with uniform upper bound of the total complex $m-$Hessian measures.

\n
Our last result focuses again on the special case when $\Omega$ is the unit ball in $\mathbb C^n$. 

\n 
\textbf{Theorem E.}
{\it Let $u\in \SHm^-(\mathbb{B}^{n})$. Assume that there exists $A>0$ such that
\begin{equation}\label{eq main3}
\lim\limits_{\de\to 0^+}\dfrac{\vol_{2n}(\{z\in\mathbb{B}^{n} : \|z\|>1-\de, u(z)<-A\de\})}{\de}=0.
\end{equation}
Then $u\in \Fm(\mathbb{B}^{n})$.}

\n
The proof is a slightly expanded version of that of Theorem 5 in \cite{DD19} where the same statement is proved
when $m=n$. The main step of our proof is to approximate from above $u$ by a collection of $m-$subharmonic
$u_{a,\ve}$ which lives on slightly smaller balls. The function $u_{a,\ve}$ is constructed by taking upper envelopes of a family generated by $u$ and a sequence of rotations.
Next, as in the proof of Theorem D, we will exploit the assumption on the volume decay of $u<-A\de$ near the boundary to get a lower estimate of $u_{a,\ve}$ in terms of some defining function for $\B^n.$ Then we will glue these data together to obtain a sequence in $\E^0_m (\Om)$ that approximate $u$ "correctly".

\medskip
\noindent
{\bf Acknowledgments}. The first named author is supported by Vietnam National Foundation for Science and Technology Development (NAFOSTED) under grant number 101.02-2019.304. 
The second named author would like to thank IMU and TWAS for supporting his PhD studies through the IMU Breakout Graduate Fellowship.
\section{Preliminaries}
In this  short section, we will review some basic technical tools that will be used in our work.
\subsection{$m$-complex Hessian measure}

Let $u$ be a locally bounded $m-$subharmonic function defined on a domain $\Om$ in $\mathbb C^n.$ Then, following Bedford and Taylor in [1], by induction we may define the $m-$complex Hessian measure of $u$ as
$$(dd^c u)^m\wedge \om^{n-m}:= dd^c (u (dd^c u)^{m-1} \wedge \om^{n-m+1}).$$
A natural problem is to define the largest subset of $SH_m^{-} (\Om)$ on which the above operator is well defined and enjoy  
the continuity property under monotone convergence. This results in the introduction of the classes $\E_m (\Om)$ and $\F_m (\Om)$ mentioned at the beginning of our article. A major tool in studying $m-$complex Hessian measures
is the following comparison principle.
\begin{proposition}
Let $u,v$ be locally bounded $m-$subharmonic function on a bounded domain $\Om$ in $\mathbb C^n.$
Suppose that 
$\liminf\limits_{z \to \partial \Om}\ (u(z)-v(z)) \ge 0.$ Then we have
$$\int\limits_{\{u<v\}} (dd^c u)^m \wedge \om^{n-m} \ge \int\limits_{\{u<v\}} (dd^c v)^m \wedge \om^{n-m}.$$
\end{proposition}
\n 
The above result can be proved exactly in the same way as Theorem 4.1 in \cite{BT} where the case $m=n$ is treated. So it will be referred to naturally as Bedford-Taylor's comparison principle.
A main consequence of this principle is the 
following useful fact that compares total complex $m-$Hessian masses of elements in $\F_m (\Om).$ 
\begin{lemma}\label{compa NP}
Let $\Omega$ be a bounded $m-$hyperconvex domain in $\mathbb C^n$ and $u,v\in\Fm(\Omega)$. Suppose that $u\geq v$ on a small \nhd\ of  $\partial\Omega$. Then
$$\int\limits_\Omega (dd^cu)^m\wedge\omega^{n-m}\leq\int\limits_\Omega (dd^cv)^m\wedge\omega^{n-m}.$$
\end{lemma}
\begin{proof}
We first consider the case when $u, v \in \E^0_m (\Om)$. Then the result can be proved by applying Proposition 2.1 to $u, \la v$ with $\la>1$ and then by letting $\la \to 1$ we reach the desired estimate. The general case can be proved by looking at the definition of $\Fm (\Om)$ as was done in the case $m=n.$
\end{proof}	
\n
More subtle aspect of $m-$subharmonic functions lies in their subextension property. 
Indeed, using the solvability of the complex $m-$Hessian equation, we have the following result about subextension of $m-$subharmonic. 
The proof follows closely the lines of \cite{CZ} where a similar statement was proved for plurisubharmonic functions. 
\begin{theorem}[\cite{HD}]\label{sub}
Let $\Omega\subset\widetilde{\Omega}\subset\mathbb C^n$ be bounded $m-$hyperconvex domains and $u\in \Fm(\Omega)$. Then,
there exists $v\in \Fm(\widetilde{\Omega})$  such that $v \leq u$ on $\Omega$ and
$$(dd^cv)^m\wedge\omega^{n-m}=1_\Omega(dd^cu)^m\wedge\omega^{n-m}\text{ on }\widetilde{\Omega}.$$
\end{theorem}
\subsection{The averaging lemma}
The aim of this subsection is to introduce a device that creates elements in Cegrell's classes by integrating with parameters a family of $m-$subharmonic functions. 
We start with a somewhat standard lemma that relaxing the pointwise convergence condition in the definition of  $\F_m (\Om)$ to almost everywhere (a.e.) convergence.
\begin{lemma}\label{ae}
Let $\Omega$ be a $m-$hyperconvex domain in $\mathbb C^n$ and $u\in \SHm^-(\Omega)$. Assume that there exists a sequence
$\{u_j\}\in\Fm (\Omega)$ such that $u_j$ converges a.e. to
$u$ and
$$\sup\limits_{j>0}\int\limits_{\Omega}(dd^cu_j)^m\wedge\omega^{n-m}<\infty.$$ 
Then $u\in\Fm(\Omega)$.
\end{lemma}
\begin{proof}
Let $\rho \in SH^{-}_m (\Om)$ be an exhaustion function for $\Om.$ For $k \ge 1$ we set
$$\tilde u_k(z):=\sup\limits_{j\geq k} (\max\{u,u_j, k\rho\})\text{ and }\ v_k:= \tilde u_k^*.$$
Then we have the following facts about $v_k$:
\begin{itemize}
\item[(i)] $v_k\in\SHm^-(\Omega)\text{ and } v_k\geq u_k\  \forall k\geq1,$
\item[(ii)] $\{v_k\}$ is decreasing and $v_k\geq u \ \forall k\geq1,$
\item[(iii)] $v_k=\tilde u_k$ a.e. and so $v_k\downarrow u$ everywhere on $\Om,$ 
\item[(iv)]	$\lim\limits_{z \to \partial \Om} v_k (z)=0.$
\end{itemize}
Here the second assertion of (iii) follows the assumptions that $u_k \to u$ a.e. and $u \in SH_m (\Om)$.
Moreover, since $u_k\in \F_m(\Omega)$, we get $\tilde v_k\in\F_m(\Omega)$. Finally, by Lemma \ref{compa NP}, we obtain
$$C:=\sup\limits_{k}\int\limits_\Omega(dd^cu_k)^m\wedge\omega^{n-m}\geq\sup\limits_{k}\int\limits_{\Omega}(dd^cv_k)^m\wedge\omega^{n-m}.$$
Thus, $u\in\F_m(\Omega)$ as desired.
\end{proof}
The averaging lemma below is perhaps of independent interest.
\begin{lemma}\label{lem 1 main2} Let $\Omega\subset\mathbb C^n$ be a bounded $m-$hyperconvex domain and $X$ be a compact metric space equipped with a probability measure $\mu$. 
Let $u: \Omega\times X\rightarrow [-\infty, 0)$ such that
\begin{itemize}
\item [(i)] For every $a\in X$, $u(., a)\in\Fm (\Omega)$ and
\begin{center}
$\int\limits_{\Omega}(dd^cu(z, a))^m\wedge\omega^{n-m} \le M$,
\end{center}
where $M>0$ is a constant.
\item[(ii)] For every $z\in\Omega$, the function $u(z,.)$ is upper semicontinuous on $X$.
\end{itemize}
	Then the following assertions hold true:
	
\n 
(a) $\tilde{u}(z):=\int\limits_X u(z,a)d\mu(a)\in\Fm (\Omega)$.

\n 
(b) $\int\limits_{\Omega}(dd^c \tilde u)^m\wedge\omega^{n-m} \le M.$
\end{lemma}
\begin{proof} For each $j\ge 1$, decompose $X$ into a finite pairwise disjoint collection of Borel sets  $U_{j,1},...,U_{j,m_j}$ having diameter less than $\frac{1}{2^j}$.
Set
$$\begin{aligned} 
u_j(z):=&\sum_{k=1}^{m_j}\mu(U_{j,k})\sup_{a\in U_{j,k}} u (z,a) \\
v_j (z,a):=& \sum_{k=1}^{m_j} {\bf 1}_{U_{j,k}} (a) \sup_{b\in U_{j,k}} u (z,b).
\end{aligned}$$
We claim that $u_j$ converges pointwise to $\tilde u$ on $\Om.$ Indeed, since $\mu$ is a probability measure we infer that
$u_j \ge \tilde u$ for every $j.$ On the other hand, for any fixed $z \in \Om,$ using the assumption (ii) and then Fatou's lemma, we obtain
$$\begin{aligned} 
\tilde u (z)&=\int\limits_X u(z,a)d\mu(a)\\
&\ge \int\limits_X \limsup\limits_{j \to \infty} v_j (z,a) d\mu(a)\\
&\ge \limsup\limits_{j \to \infty} \int\limits_X v_j (z,a) d\mu(a)=\limsup\limits_{j \to \infty} u_j (z).
\end{aligned}$$
Thus, we have indeed $u_j \to \tilde u$ pointwise on $\Om$ as claimed.
So $u_j^* \to \tilde u$ a.e. on $\Om$ since $u_j^*=u_j$ a.e. on $\Om.$
It now remains to bound the complex $m-$Hessian measures of $u_j^*.$ For this,
we choose $a_{j,k} \in U_{j,k}$ for $1 \le k \le m_j.$ Then
$$u_j^* \ge u_j \ge \tilde u_j:=\sum_{k=1}^{m_j}\mu(U_{j,k}) u(\cdot, a_{j,k}) \ \forall j \ge 1.$$
Since $\F_m (\Om)$ is a convex cone, we infer that $\tilde u_j\in\F_m(\Omega)$, and hence  $u_j^*\in\F_m(\Omega)$. Moreover, by Lemma 2.2 we obtain, for $j \ge 1$,
$$\begin{aligned} 
&\int\limits_{\Omega}(dd^c u_j^*)^m\wedge\omega^{n-m} \leq \int\limits_{\Omega}(dd^c \tilde u_j)^m\wedge\omega^{n-m}\\
=&\int\limits_\Omega\Big[\Big(\sum_{k=1}^{m_j}\mu(U_{j,k})dd^c u(z,a_{j,k})\Big)\Big]^m\wedge\omega^{n-m}\\
=&\sum_{k_1+...+k_{m_j}=m}\frac{m!}{k_1!k_2!...k_{m_j}!}\prod_{l=1}^{m_j}\mu(U_{j,l})^{k_l}\int\limits_\Omega\Big(dd^cu(z,a_{j,1})\Big)^{k_1}\wedge...
 \wedge\Big(dd^cu(z,a_{j,m_j})\Big)^{k_{m_j}}\wedge\omega^{n-m}\end{aligned}$$ 
Therefore, by appling a Cegrell-H\"older's type inequality in the fourth estimate (see Proposition 3.3 in \cite{HP17}), we have, for $j \ge 1$,
$$\begin{aligned} 
\int\limits_{\Omega}(dd^c u_j^*)^m\wedge\omega^{n-m}\leq&\sum_{k_1+...+k_{m_j}=m}\frac{m!}{k_1!k_2!...k_{m_j}!}\prod_{l=1}^{m_j}\mu(U_{j,l})^{k_l}
\Big[\int\limits_\Omega\Big(dd^cu(z,a_{j,l})\Big)^{m}\wedge\omega^{n-m}\Big]^{k_l/m}\\
\leq&M\sum_{k_1+...+k_{m_j}=m}\frac{m!}{k_1!k_2!...k_{m_j}!}\prod_{l=1}^{m_j}\mu(U_{j,l})^{k_l}\\
=&M(\sum_{l=1}^{m_j}\mu(U_{j,l}))^m=M.
\end{aligned}$$
So, by Lemma 2.4, we conclude that $u \in \F_m (\Om).$
\end{proof}
\section{Proofs of the results}
In this section we will provide detailed proofs of the results that are announced at the beginning of the article.
We first deal with Theorem A. The main technique is the classical Bedford-Taylor comparison principle and 
the structure of Cegrell classes that involved. 

\begin{proof} [Proof of Theorem A]
(a) Let  $u_j \in \Eom (\Om)$ be a sequence satisfying $u_j \downarrow u$ and
$$\lim\limits_{j \to \infty} \int\limits_\Om (dd^c u_j)^m\wedge\omega^{n-m}=\int\limits_{\Om} (dd^c u)^m\wedge\omega^{n-m}.$$
Fix an open subset $\Om' \Subset \Om$, we can find $\rho' \in \Fm(\Om)$ with $\rho'|_{\Om'}=\rho.$
Then we note the inclusion
$$\Om(u_j,\ve,\de):= \{z \in \Om: u_j(z)<-\ve, \rho'>-\de\} \subset \big \{\rho'>\fr{\de}{\ve} u_j \big \}.$$
Thus, by using Bedford-Taylor's comparison principle, we get the following chain of estimates
$$
\begin{aligned} \big (\fr{\de}{\ve}\big )^m\int\limits_{\Om} (dd^c u_j)^m\wedge\omega^{n-m} &\ge \int\limits_{\big \{\rho'>\fr{\de}{\ve} u_j \big \}} \big (\fr{\de}{\ve}\big )^m (dd^c u_j)^m\wedge\omega^{n-m}\\ 
&\ge \int\limits_{\big \{\rho>\fr{\de}{\ve} u_j \big \}} (dd^c \rho')^m\wedge\omega^{n-m}\\ 
&\ge \int\limits_{\Om(u_j,\ve,\de)} (dd^c \rho')^m\wedge\omega^{n-m}\\
&\ge \int\limits_{\Om(u_j,\ve,\de) \cap \Om'} (dd^c \rho)^m\wedge\omega^{n-m}
\end{aligned}$$
	Since $\Om (u_j,\ve,\de) \cap \Om' \uparrow \Om_{u,\ve,\de} \cap \Om'$, by letting $j \to \infty$ and then $\Om' \uparrow \Om$ we obtain the desired estimate.
	
\n 
(b) For each $a \in (0,1)$ we set
$$\rho_a:=-(-\rho)^a.$$
Then, by a direct computation, we obtain the following identity in the sense of currents
$$dd^c \rho_a= a(1-a)(-\rho)^{a-2} d\rho \wedge d^c \rho + a(-\rho)^{a-1}dd^c \rho.$$
Then $\rho_a$ is a negative locally bounded $m-$\psh\ function on $\Om$. 
Moreover,
$$(dd^c \rho_a)^m\wedge\omega^{n-m} \ge ma^m (1-a) (-\rho)^{m(a-1)-1} d\rho \wedge d^c \rho \wedge (dd^c \rho)^{m-1}\wedge\omega^{n-m}.$$
Since $0<-\rho<\de$ on $\Om_{u,\ve,\de}$, we may combine the above inequality and
the estimate in (a) to obtain
$$\begin{aligned}
&ma^m (1-a) \de^{m(a-1)-1} \int\limits_{\Om_{\ve,\de}} d\rho \wedge d^c \rho \wedge (dd^c \rho)^{m-1}\wedge\omega^{n-m}\\ 
&\le \int\limits_{\Om_{u,\ve,\de}} (dd^c \rho_a)^m\wedge\omega^{n-m}\\
&=\int\limits_{\{u<-\ve,\rho_a>-\de^a\}} (dd^c \rho_a)^m\wedge\omega^{n-m}\\
& \le  \big (\fr{\de^a}{\ve} \big)^m \int\limits_\Om (dd^c u)^m\wedge\omega^{n-m}.
\end{aligned}$$
Now our inequality follows by rearranging these estimates and taking $a=\fr{m}{m+1}$. \end{proof}
\vskip 0,2cm
\n
It is natural to ask if the following converse to Theorem A is true.
\begin{question}\label{Ques1}
Let $u$ be a negative $m-$subharmonic function on a bounded hyperconvex domain $\Om.$
Suppose that there exists $A>0$ such that for all $\ve>0, \de>0$ and for all $\rho \in \Em(\Om)$ we have
$$\int\limits_{\Om_{u,\ve,\de}} (dd^c \rho)^m\wedge\omega^{n-m} \le A\big (\fr{\de}{\ve} \big)^m.$$
Does $u$ belong to $\Fm (\Om)?$
\end{question}
\n
Theorem E is, thus, an attempt, to answer this question in the affirmative 
when $\Om$ is the unit ball in $\mathbb C^n.$
The following result follows directly from Theorem A (a).
\begin{corollary}\label{Cor2}
Let $\Om$ be a bounded $B-$regular domain, i.e., there exists a negative \psh\ exhaustion function 
$\rho$ on $\Om$ satisfying $dd^c \rho \ge \om.$ Then for all $u \in \Fm(\Om)$ we have
$$\vol_{2n} (\Om_{u,\ve,\de}) \le \fr{\de^{m}}{\ve^m} \int\limits_\Om (dd^c u)^m\wedge\omega^{n-m}.$$
\end{corollary}
\n 
Notice that we are using here the notion of $B-$regular domains taken from the seminal work \cite{Sib}.
Under a stronger assumption on convexity and smoothness of $\Om$ we may refine the above estimate as follows.
\begin{corollary}\label{Cor3}
Let $\Om$ be a bounded strictly $m-$pseudoconvex domain with $C^2-$smooth boundary.
For $\de>0$ and $u \in \Fm(\Om)$ we set
$$\Om_u(\ve,\de):=\{z \in \Om: u(z)<-\ve, d(z,\partial \Om)<\de\},$$
	where $d$ is the distance function. Then there exist $\delta_0=\delta_0(\Om)>0$ and $C=C(\Om,\delta_0,n)>0$ such that for all $u \in \Fm(\Om)$, $\de\in (0,\delta_0)$ and $\ve>0$ we have
	$$\vol_{2n} (\Om_u (\ve,\de)) \le C\fr{\de^{m+1}}{\ve^m} \int\limits_\Om (dd^c u)^m\wedge\omega^{n-m}.$$
\end{corollary}
\begin{proof}
	Let $\rho$ be an arbitrary strictly $m-$\psh\ functions on a \nhd\ of $\overline \Omega$ that defines $\Om.$
	Then we can find a positive constant $\de_0$ depending on $\Om$ such that
	$$d\rho\neq 0\text{ on }\{z\in\Om:\ d(z,\partial\Om)\leq\de_0\}.$$
	Thus, on $\Om$,
	$$d\rho \wedge d^c\rho \wedge (dd^c \rho)^{m-1}\wedge\omega^{n-m} \ge A d\rho \wedge d^c\rho \wedge \om^{n-1}=A\Vert \text{grad}\ \rho \Vert^2 \om^n,$$
	for some constant $A>0$. Therefore, since $\Vert \text{grad}\ \rho \Vert$ is bounded from below by a positive constant, we have, on $\{z\in\Om:\ d(z,\partial\Om)\leq\de_0\}$,
	$$d\rho \wedge d^c\rho \wedge (dd^c \rho)^{m-1}\wedge\omega^{n-m} \ge C'\om^n,$$
	for some constant $C'$. It follows that
	$$\int\limits_{\Om_u(\ve,\de)}d\rho \wedge d^c\rho \wedge (dd^c \rho)^{m-1}\wedge\omega^{n-m}\geq C'\int\limits_{\Om_u(\ve,\de)}\om^n,$$
	for all $\ve>0$ and $\de\in (0,\de_0)$. The desired estimate follows by combining this with Theorem A(b).
\end{proof}
The following question is curiously open to us.
\begin{question}\label{Ques2}
	Let $\Om$ be a $\mathbb C^2$ smooth strictly pseudoconvex. Is there a smooth defining strictly $m-$\psh\ function for $\Om$ whose gradient is non-vanishing {\it entirely} on $\Om?$
\end{question}
\n 
If the answer to the above question is positive then the constant given in Corollary~\ref{Cor2} can be chosen to be independent of $\ve_0.$
Regarding boundary behavior of $\Fm (\Om)$, we have the following result which will also be used in the proof of Proposition 3.6.
\begin{proposition}\label{Prop2}
	Let $u, \rho \in \Fm (\Om).$ Then we have
	$$\liminf_{z \to  \partial \Om} \frac{u(z)}{\rho(z)} \le M,$$
	where
	$$M:= \Big ( \frac{\int\limits_{\Om} (dd^c u)^m\wedge\omega^{n-m}}{\int\limits_{\Om} (dd^c \rho)^m\wedge\omega^{n-m}}\Big )^{1/m} \in (0,\infty).$$
\end{proposition}
\begin{proof}
	Fix $j \ge 1$. We claim that
	$$M_j:=M+\fr1{j} \ge \liminf\limits_{z \to  \partial \Om} \frac{u(z)}{\rho(z)}.$$
	Assume the contrary holds, then we have $u \le (M+\fr1{2j}) \rho$ on a small \nhd\ of $\partial \Om.$
	Thus $$u \le v_j:=\max \{u, (M+\fr1{2j})\rho\} \in \F (\Om)$$ 
	and $v_j=(M+\fr1{2j}) \rho$ near $\partial \Om.$
	Then by the comparison principle we obtain
	\begin{align*}
	M^m \int\limits_{\Om} (dd^c \rho)^m\wedge\omega^{n-m}&=\int\limits_{\Om} (dd^c u)^m\wedge\omega^{n-m}\\
	&\ge \int\limits_{\Om} (dd^c v_j)^m\wedge\omega^{n-m}\\
	&=(M+\fr1{2j})^m 
	\int\limits_{\Om} (dd^c \rho)^m\wedge\omega^{n-m}.
	\end{align*}
	Here we used Stokes' theorem for the last equality. So we obtain a contradiction and thus the claim follows. By letting $j \to \infty$, we obtain the desired conclusion.
\end{proof}
\n
The above result can be used to characterized radial elements in $\Fm (\Om)$ when $\Om$ is a {\it ball} in $\mathbb C^n,$ a problem of independent interest.

\n 
{\bf A word of caution}: From now on we always use $a_n$ (which may change from line to line) to mean an absolute constant that depends only on $n$.
\begin{proposition}\label{Prop3}
Let $u \in \SHm^{-} (\mathbb B^n (0,r))$ be a radial function. Then the following conditions are equivalent.
	
\n 
(a) $u\in \Fm (\mathbb B^n (0,r));$
	
\n 
(b) $\sup\limits_{0 \le t<r} \frac{u(t)}{t-r} \le a_n M(r),$
where 
$$M(r):=\fr1{r}\big (\int\limits_{\mathbb B^n (0,r)} (dd^c u)^m\wedge\omega^{n-m}\big )^{1/m}.$$ 
\end{proposition}
\begin{proof}
	If $(b)$ holds then 
	$u(z) \ge a_n (M+1) \Big(1-\frac{r^{2(n/m-1)}}{|z|^{2(n/m-1)}}\Big)$ on a small \nhd\ of $\partial \mathbb B^n (0,r).$ This implies (a) since the function on the right-hand side belongs to $\Fm (\mathbb B^n (0,r))$. 
	On the other hand, if (a) is  true then we first apply Proposition~\ref{Prop2} to $\rho(z):= 1-\frac{r^{2(n/m-1)}}{|z|^{2(n/m-1)}}$ to obtain
	\begin{equation}\label{equa1}
	\liminf\limits_{t \to r} \frac{u(t)}{t-r} \le a_n M(r).
	\end{equation}
	Now suppose (b) is false then there exists $t_0 \in (0,r)$ and $\la>a_nM(r)$such that
	$$u(t_0)<\la(t_0-r).$$
	Since  $\lim\limits_{t \uparrow r} u(t)=0,$ we may apply convexity of $u$ on $[t_0, r)$
	to conclude that 
	$$u(t)<\la(t-r).$$
	This is a contradiction to (\ref{equa1}). We are done. 
\end{proof}
We now proceed to the proof of Theorem B. The proof requires 
the following auxiliary result, which might be of independent interest.
\begin{lemma}\label{Lem4}
	Let $u \in \Fm (\mathbb B^n (0,r)).$ Then for all $\de \in (0,r)$ we have
	$$\fr1{\de}\int\limits_{\vert z\vert=r-\delta} u(z) d\sigma (z) \ge 
	-a_nr^{2n-2}\big (\int\limits_{\Om} (dd^c u)^m\wedge\omega^{n-m}\big )^{1/m}.$$
\end{lemma}
\begin{proof}
We are going to use Proposition~\ref{Prop3} and 
a symmetrization trick as in \cite{DD19}.
Define $\tilde u$ as in \cite{DD19}. Note that
	$\tilde u$ is radial and belongs to $\Fm (\Om).$ Moreover
	$\int \limits_\Om (dd^c \tilde u)^m\wedge\omega^{n-m} \le \int \limits_\Om (dd^c u)^m\wedge\omega^{n-m}.$ 
	It then follows from Proposition \ref{Prop3} that 
	$$\tilde u(z) \ge (\vert z\vert-r)a_n M(r) \ \forall z \in \mathbb B^n (0,r).$$
	This implies that 
	$$\fr1{(r-\de)^{2n-1}}\int\limits_{\vert z\vert=r-\delta} u(z) d\sigma (z) \ge -\de a_nM(r).$$
	After rearranging the above estimate, we get our desired inequality.
\end{proof}
\begin{proof}[Proof of Theorem B]
The proof is splitted into two steps.

\n
{\it Step 1.} 
We will show that for $r \in (0,d(\eta))$ we have
$$\fr1{d(\eta)-r} \int\limits_{\{\{\vert z-\eta\vert=r\}\cap \Om\}} u(z) d\sigma(z)\ge -a_n d^{2n-2}\big (\int\limits_{\Om} (dd^c u)^m\wedge\omega^{n-m}\big )^{1/m}.$$
Consider the open ball $\Om':=\mathbb B(\eta, d(\eta)).$
Then $\Omega \subset \Om'$ and $\xi \in \partial \Om' \cap \partial \Om.$ By a sub-extension 
result \cite{HD}, we can find $u' \in \Fm (\Om')$ such that
$u' \le u$ on $\Om$ but $(dd^c u')^m\wedge\omega^{n-m} =\chi_{\Om} (dd^c u)^m\wedge\omega^{n-m}.$ Note that this method is inspired from \cite{CZ}.
Thus, by Lemma \ref{Lem4}, we obtain
$$\begin{aligned}
\fr1{d(\eta)-r} \int\limits_{\{\{\vert z-\eta\vert=r\}\cap \Om\}} u(z) d\sigma(z)
&\ge \fr1{d(\eta)-r} \int\limits_{\{\{\vert z-\eta\vert=r\}\}} u'(z) d\sigma(z)\\
&\ge -a_n d(\eta)^{2n-2}\big (\int\limits_{\Om'} (dd^c u')^m\wedge\omega^{n-m}\big )^{1/m} \\
&\ge -a_n d^{2n-2}\big (\int\limits_{\Om} (dd^c u)^m\wedge\omega^{n-m}\big )^{1/m}.
\end{aligned}$$
Therefore, we obtain the required estimate.

\n
{\it Step 2.} Completion of the proof. By the result obtained in the first step, such that for $t>0$ for all $r \in (0, d(\eta))$ we have
$$\vol_{2n-1} \{z \in \Om: u(z)<-t, \vert z-\eta\vert=r\}
\le \fr{d(\eta)-r}{t}a_n d^{2n-2}\big (\int\limits_{\Om} (dd^c u)^m\wedge\omega^{n-m}\big )^{1/m}.$$
Thus, for $\de \in (0, d(\eta))$, we obtain
$$\begin{aligned}
&\vol_{2n} \{z \in \Om: u(z)<-t, d(\eta)-\delta <\vert z-\eta \vert <d(\eta)\} \\
&=\int\limits_{d(\eta)-\delta}^{d(\eta)} \vol_{2n-1} \{z \in \Om: u(z)<-t, \vert z-\eta\vert=\la\}d\la\\
&\le a_n d^{2n-2}\big (\int\limits_{\Om} (dd^c u)^m\wedge\omega^{n-m}\big )^{1/m} \fr{\de^2}{t}.
\end{aligned}$$
The proof is thereby completed.
\end{proof}
\n
Concerning the geometry of the domain $\Om$ in Theorem D, we have the following question.
\begin{question}
Let $\Om$ be a bounded domain with $\mathcal C^2$ smooth boundary. Assume that $\Om$ is $m-$hyperconvex. Does $\Om$ admits a $\mathcal C^2$ smooth defining function which is $m-$subharmonic on $\Om$?
\end{question}
\n
If $m=n$ then the answer is yes according to a famous result of Diederich and Fornaess. Next we proceed to the

\begin{proof} [Proof of Theorem D]
By multiplying $\rho$ with a small positive constant we can assume $\rho >-1$ on $\Om.$
Since the gradient of $\rho$ is nowhere zero on $\partial \Om,$ using the implicit function theorem, we can find positive constants $C_1,C_2$
such that
\begin{equation}\label{inequa3}
	C_1d(z,\partial\Omega)\leq-\rho(z)\leq C_2d(z,\partial\Omega) \ \forall z\in\Omega.
\end{equation}
We consider two cases

\n
{\it Case 1.} $u\geq a\rho$ in $\Omega$ for some $a>0$. For $\ve>0$, we let  $$\Omega_\ve:=\{z\in\Omega:\text{d}(z,\partial\Omega)>\ve\}.$$
We then define on  $\Omega_\ve$ the function
$$u_\ve(z):=\frac{1}{c_n \ve^{2n}}\int\limits_{\B(z,\ve)}u(\xi)dV(\xi)= \frac{1}{c_n \ve^{2n}}\int\limits_{\B(0,\ve)}u(z+\xi)dV(\xi),$$
where $dV$ denote the Lebesgue measure on $\mathbb C^n$ and $c_n$ is the volume of unit ball in $\mathbb C^n$.
We have $u_\ve\in\SHm^-(\Omega_\ve)$ and $u_\ve\downarrow u$ when $\ve\downarrow0$.
Our key step is to estimate $u_\ve$ from below by a {\it fixed} multiple of $\rho$ for $\ve$ small enough.
To this end, for $\de>1$ and $0<\ve_0<1$, we consider the annulus $z\in\Omega$ such that
\begin{equation}\label{inequa8}
\Om^{\de,\ve_0}:= \{z \in \Om:	\ve_0<\ve=\text{d}(z,\partial\Omega)<2\de^2\ve_0\}.
\end{equation} 
So for $z \in \Om^{\de,\ve_0}$ we have
$$u_{\ve_0}(z)=\frac{1}{c_n \ve_0^{2n}}\Big(\int\limits_{B_1}u(\xi)dV(\xi)+\int\limits_{B_2}u(\xi)dV(\xi)\Big),$$
where 
$$B_1:= \{\xi\in\B(z,\ve_0): u(\xi)<-A(\ve+\ve_0)\}, B_2:=\B(z,\ve_0)\setminus B_1.$$  
Since $u\geq a\rho$ in $\Omega,$ using (\ref{inequa3}) we obtain
$$u_{\ve_0}(z)\geq\frac{1}{c_n \ve_0^{2n}}\Big(\int\limits_{B_1}-aC_2\text{d}(\xi,\partial\Omega) dV(\xi)+\int\limits_{B_2}-A(\ve+\ve_0)dV(\xi)\Big).$$
Observe that
$$B_1\subset\{\xi\in\Omega: \ \text{d}(\xi,\partial\Omega)<\ve+\ve_0,\ u(\xi)<-A(\ve+\ve_0)\}.$$
So by the assumption of the theorem we obtain
$$\vol_{2n}(B_1) \le C(\ve+\ve_0)^\alpha.$$
Combining this with (\ref{inequa8}), we obtain for $z \in \Om^{\de,\ve_0}$ the lower estimate for $u_{\ve_0}$
\begin{align*}
u_{\ve_0}(z)&\geq\frac{-aCC_2}{c_n \ve_0^{2n}}(\ve+\ve_0)^{\alpha+1}-A(\ve+\ve_0)\\
&\geq\frac{-2aCC_2}{c_n \ve_0^{2n}}(\ve+\ve_0)^{\alpha}\ve-2A\ve\\
&\geq\frac{-2aCC_2}{c_n}(2\de^2+1)^\alpha \ve_0^{\alpha-2n}\ve-2A\ve.
\end{align*}
Thus, by applying again (\ref{inequa3}) we get $$u_{\ve_0}(z)\geq \Big[\frac{2aCC_2}{c_nC_1}(2\de^2+1)^\alpha \ve_0^{\alpha-2n}+\frac{2A}{C_1}\Big]\rho(z).$$
Since $\alpha-2n>0$, the first term inside the bracket tends to $0$ when $\ve_0$ tends to $0$. Hence, there exists $\ve^*_0>0$ depending only on $a$ such that
\begin{equation}\label{inequa4}
	u_{\ve_0}\geq C_3\rho \text{ in }\Omega_{\ve_0},\text{ for all } \ve_0<\ve^*_0
\end{equation}
where $C_3:=\frac{2A}{C_1}+1$. 
Set
$$\de:= 2\frac{C_2}{C_1}\ \text{and}\ \lambda:= \frac{C_3}{\frac{1}{\de C_2}-\frac{1}{\de^2C_1}}.$$
For $\ve_0<\ve^*_0$, we will estimate $u_{\ve_0}(z)-\lambda \ve_0$ from above and from below
on $\partial\Omega_{\de\ve_0}$ and $\partial\Omega_{\de^2\ve_0}$ respectively. 
To this end, we first use (\ref{inequa3}) to obtain
\begin{equation}\label{inequa5}
	u_{\ve_0}(z)-\lambda \ve_0=u_{\ve_0}(z)-\frac{\lambda}{\de} \text{d}(z,\partial\Omega)\leq\frac{\lambda}{\de C_2}\rho(z)\ \text{for}\ z\in\partial\Omega_{\de \ve_0}.
\end{equation}
By (\ref{inequa4}) and (\ref{inequa3}), we have
\begin{equation}\label{inequa6}
	u_{\ve_0}(z)-\lambda \ve_0=u_{\ve_0}(z)-\frac{\lambda}{\de^2} \text{d}(z,\partial\Omega)\geq\Big(C_3+\frac{\lambda}{\de^2C_1}\Big)\rho(z)\ \text{for}\ z\in\partial\Omega_{\de^2\ve_0}.
\end{equation}
Combining (\ref{inequa5}), (\ref{inequa6}) and noting that 
$$\frac{\lambda}{\de C_2}=C_3+\frac{\lambda}{\de^2C_1},$$ we derive for $\ve_0<\ve^*_0$ the following estimates
$$\begin{cases}
u_{\ve_0}(z)-\lambda \ve_0\leq\beta\rho(z)\ \text{for}\ z\in\partial\Omega_{\de \ve_0}\\
u_{\ve_0}(z)-\lambda \ve_0\geq\beta\rho(z)\ \text{for}\ z\in\partial\Omega_{\de^2\ve_0}
\end{cases},$$
where $\beta=\frac{\lambda}{\de C_2}$. Now, for $\ve_0<\ve^*_0$, we consider
$$\widetilde{u}_{\ve_0}(z)=\begin{cases}
\beta\rho,&\text{in}\ \Omega\backslash\Omega_{\de \ve_0}\\
\max(\beta\rho,u_{\ve_0}-\lambda \ve_0),&\text{in}\ \Omega_{\de\ve_0}\backslash\overline{\Omega_{\de^2\ve_0}}\\
u_{\ve_0}-\lambda \ve_0,&\text{in}\ \overline{\Omega_{\de^2\ve_0}}
\end{cases}.$$
We have $\widetilde{u}_{\ve_0}\in\E_m^0(\Omega)$, $\widetilde{u}_{\ve_0}\downarrow u$ when $\ve_0\downarrow0$ and by the comparison principle, we have
$$\int\limits_\Omega(dd^c\widetilde{u}_{\ve_0})^m\wedge\omega^{n-m}\leq \beta^{2m}\int\limits_\Omega(dd^c\rho)^m\wedge\omega^{n-m},$$
for $\ve_0$ small enough. Therefore $u\in\Fm(\Omega)$ as we want.

\n
{\it Case 2.} Now we treat the general case. For $N \ge 1$, we set
$u_N:=\max \{u, N\rho\}.$ Then $u_N \in \Fm(\Omega)$ and $u_N \downarrow u.$ By the result obtained in Case 1, we have
$$\sup\limits_{N\ge 1} \int\limits_\Omega(dd^cu_N)^m\wedge\omega^{n-m}\leq \beta^{2m}\int\limits_\Omega(dd^c\rho)^m\wedge\omega^{n-m}.$$
Therefore $u\in\Fm(\Omega).$ The proof is thereby completed.
\end{proof}	
\n
\begin{proof} [Proof of Theorem E]
	Denote by $U(n)$ the set of unitary transformations from $\mathbb C^n$ to $\mathbb C^n.$
%For every $0<a<1$, we set 
%$$S_a:= \{\phi\in U(n): \|\phi-Id\|<a \}.$$
%For $\ve>0, a<1$ and $z\in\B^{n}_{1-\ve}:=\{w\in\C^n: \|w\|<1-\ve \}$, we define
For $0<a<1,\ \ve>0$ and $z\in\B^{n}_{1-\ve}:=\{w\in\mathbb C^n: \|w\|<1-\ve \}$, we define
\begin{center}
$u_{a,\ve}(z):=(\sup\{u((1+r)\phi(z)): \phi\in S_a, 0\leq r\leq\ve \})^*$,
\end{center}
where $S_a:= \{\phi\in U(n): \|\phi-Id\|<a \}.$
Since $m-$subharmonicity is preserved under unitary transformations, we infer that
$u_{a, \ve}$ is $m-$subharmonic on $\B^{n}_{1-\ve}$. Moreover, by upper-semicontinuity of $u$ we obtain
\begin{equation}\label{eq1 proof3}
\lim\limits_{\max (a,\ve)\to 0^+} u_{a, \ve}(z)=u(z), \ \forall z\in\Omega.
\end{equation}
We also note that if $z\neq 0$ then
\begin{equation}\label{eq2 proof3}
u_{a,\ve}(z):=(\sup\{u(\xi): \xi\in B_{a,\ve, z} \})^*,
\end{equation}
where
\begin{center}
$B_{a,\ve, z}:= \Big \{\xi\in\mathbb C^n: \|\dfrac{z}{\|z\|}-\dfrac{\xi}{\|\xi\|}\|<a, 
\|z\|\leq\|\xi\|\leq(1+\ve)\|z\|  \Big \}.$
\end{center}
Next we observe that there exist positive constants $C_1, C_2$ which do not depend on $a \in (0,1/2),\ve>0$ and $\xi$ such that
\begin{equation}\label{eq3 proof3}
C_1 a^{2n-1}\ve<\vol_{2n}(B_{a,\ve, z})<C_2a^{2n-1}\ve.
\end{equation}
On the other hand, by the assumption \eqref{eq main3} we deduce that for $0<a<1/2$, there exists $\ve_a \in (0,a)$
such that
$$\vol_{2n}\{\xi\in\B^{2n}:\|\xi\|>1-3\ve,u(\xi)<-3A\ve\}<C_1 a^{2n-1}\ve, \ \forall \ve \in (0,\fr{\ve_a}3).$$
Hence, by \eqref{eq3 proof3}, we have, for every $3\ve\geq 1-\|z\|\geq \ve$,
$$B_{a,\ve, z}\nsubseteq \{\xi\in\B^{n}:\|\xi\|>1-3\ve,u(\xi)<-3A\ve\}.$$
Combining this fact with \eqref{eq2 proof3} we conclude 	
that for $a \in (0,1/2)$, there exists
$\ve_a>0$ such that, for every $\ve_a>3\ve\geq 1-\|z\|\geq\ve>0$,
we have the following crucial estimate 
\begin{equation}
u_{a, \ve}(z)\geq -3A\ve.
	\end{equation}
Now for $a \in (0,1/2)$ and $0<\ve <\ve_a/3$, consider the following function
$$\tilde{u}_{a, \ve}(z):=
\begin{cases}
3A(-1+|z|^2) & \quad 1-\ve\leq \|z\| < 1,\\
\max\{3A(-1+|z|^2), u_{a, \ve}(z)-6A\ve \} &
		\quad 1-3\ve\leq \|z\|\leq 1-\ve,\\
		u_{a, \ve}(z)-6A\ve & \quad \|z\|\leq 1-3\ve.
\end{cases}$$
Then $\lim\limits_{z \to \partial \B^n} \tilde{u}_{a, \ve} (z)=0,$ and by (3.10) $\tilde{u}_{a, \ve} \in SH_m^{-} (\B^n).$ Furthermore, by Lemma 2.2 we get
\begin{center}
$\int\limits_{\B^{n}}(dd^c\tilde{u}_{a, \ve})^m\wedge\omega^{n-m}
=(3A)^m\int\limits_{\B^{n}} \omega^{n}<\infty$.
\end{center}
In particular $\tilde{u}_{a, \ve}\in\E^0_m(\B^{n})$.	
Finally, for $j \ge 1$, we consider 
$u_j:=\widetilde{u}_{2^{-j}, \fr{\ve_{2^{-j}}}3}$. By \eqref{eq1 proof3}, we have $u_j \to u$  pointwise on 
$\Om$. Moreover $\sup\limits_j\int\limits_{\B^{n}}(dd^c\tilde{u}_{j})^m\wedge\omega^{n-m}<\infty$, then by Lemma \ref{ae}, we have $u\in\F_m (\Omega)$ as desired.
\end{proof}

\end{document}